\documentclass[reqno, 10pt]{amsart}
\usepackage{color}
\usepackage{amsmath}
\usepackage{amsfonts}
\usepackage{amssymb}
\usepackage{amsthm}
\usepackage{newlfont}
\usepackage{graphicx}

\newtheorem{thm}{Theorem}[section]

\newtheorem{lem}[thm]{Lemma}
\newtheorem{prop}[thm]{Proposition}
\theoremstyle{definition}

\theoremstyle{remark}
\newtheorem{rem}[thm]{Remark}

\numberwithin{equation}{section}
\numberwithin{thm}{section}

\newcommand{\ed}{\end {document}}

\newcounter{smalllist}

\title[optimal convergence rates for turbulent flow equations]{Optimal convergence rates for the three-dimensional turbulent flow equations}

\author[D. Bian]{Dongfen Bian}
\address{The Graduate School of China Academy of Engineering Physics,
P. O. Box 2101,\ Beijing 100088,\ PR China,}%
\email{bian\_dongfen@mail.com}
\author[B. Guo]{Boling Guo}
\address{Institute of Applied Physics and Computational Mathematics,
 P. O. Box 8009,\ Beijing 100088,\ PR China.}
\email{gbl@iapcm.ac.cn}

\begin{document}
\maketitle
\begin{abstract}
In this paper we are concerned with the convergence rate of solutions to
 the three-dimensional turbulent flow equations.
 By combining the $L^p$-$L^q$ estimates for the linearized equations and an elaborate energy method, the convergence rates are obtained in various norms for the solution to the equilibrium state in the whole space, when
 the initial perturbation of the equilibrium state is small in $H^3$-framework. More precisely, the optimal convergence rates of the solutions and its first order derivatives in $L^2$-norm are obtained when the $L^p$-norm of the perturbation is bounded for some $p\in[1,\frac{6}{5})$.

\end{abstract}
 \vspace{.2in} {\bf  {Key words and phrases:}}{
Turbulent flow equations; $k$-$\varepsilon$ model; optimal convergence rate; energy estimates.}

\section{Introduction}

We consider in this work the turbulent flow equations for compressible flows on $\mathbb{R}^3$,
\begin{align}\label{MHD}
\begin{cases}
\rho_t+ \mbox{div}(\rho u)=0,\\
(\rho u)_t+ \mbox{div}(\rho u\otimes u)-\Delta u-\nabla\mbox{div}u+\nabla p=-\frac{2}{3}\nabla(\rho k),\\
(\rho h)_t+\mbox{div}(\rho u h)-\Delta h=\frac{Dp}{Dt}+S_k,\\
(\rho k)_t+ \mbox{div}(\rho u k)-\Delta k=G-\rho \varepsilon,\\
(\rho \varepsilon)_t+ \mbox{div}(\rho u \varepsilon)-\Delta\varepsilon
=\frac{C_1G \varepsilon}{k}-\frac{C_2 \rho \varepsilon^2}{k},\\
(\rho, u, h, k,\varepsilon)(x,t)|_{t=0}=(\rho_0(x), u_0(x), h_0(x), k_0(x), \varepsilon_0(x)),
\end{cases}
\end{align}
 with $S_k=[\mu(\frac{\partial u^i}{\partial x_j}+\frac{\partial u^j}{\partial x_i})
 -\frac{2}{3}\delta_{ij}\mu\frac{\partial u^k}{\partial x_k}]\frac{\partial u^i}{\partial x_j}
 +\frac{\mu_t}{\rho^2}\frac{\partial p}{\partial x_j}\frac{\partial \rho}{\partial x_j}$, $G
 =\frac{\partial u^i}{\partial x_j}[\mu_e(\frac{\partial u^i}{\partial x_j}
 +\frac{\partial u^j}{\partial x_i})-\frac{2}{3}\delta_{ij}(\rho k+\mu_e\frac{\partial u^k}{\partial x_k})]$,
  where $\delta_{ij}$ is given by $\delta_{ij}=0$ if $i\neq j$, $\delta_{ij}=1$ if $i=j$,
   dynamic viscosity $\mu$ and the eddy viscosity $\mu_t$ are positive constants satisfying $\mu+\mu_t=\mu_e$, and $C_1$,  $C_2$ are also two adjustable positive constants.

 Here $\rho,\  u,\  h,\ k$ and $\varepsilon$
 denote the density, velocity, total enthalpy, turbulent kinetic energy and rate of viscous dissipation, respectively.
  The pressure $p$ is a smooth function of $\rho$. In this paper, without loss of generality,
  we have renormalized some constants to be 1. The system (\ref{MHD}) is formed by combining effect of turbulence on
 time-averaged Navier-Stokes equations with the $k$-$\varepsilon$ model equations.

 All flows encountered in engineering practice, both simple ones such as two-dimensional jets, wakes, pipe flows and flat plate boundary layers and more complicated three-dimensional ones, become unstable above a certain Reynolds number. At low Reynolds numbers flows are laminar. Flows in the laminar regime are described by the continuity and Navier-Stokes equations which have been studied by many people \cite{Abidi1, can, can2, can3, c1, d1, d5, d7, f1, giga1, giga2, giga3, iwashita, kozono, kato, o-v, lions1, Nishita, mey, xin}.
 At high Reynolds numbers flows are observed to become turbulent. A chaotic and random state of motion develops in which the velocity and pressure change continuously with time within substantial regions of flow. More precisely, at values of the Reynolds number above $Re_{crit}$ a complicated series of evens takes place which eventually leads to a radical change of the flow character. In the final state the flow behavior is random and chaotic. The motion becomes intrinsically unsteady even with constant imposed boundary conditions. The velocity and all other flow properties vary in a random and chaotic way.
 Turbulence stands out as a prototype of multi-scale phenomenon that occurs in nature.
 It involves wide ranges of
 spatial and temporal scales which makes it very difficult to study analytically
 and prohibitively expensive to simulate computationally.
 Many, if not most, flows of engineering significance are turbulent, so the
turbulent flow regime is not just of theoretical interest. Up to now, although many physicists and mathematicians studied turbulent flows, there are not any general theory suitable for them.
Fluid engineers need access to viable tools capable of representing the effects of turbulence.

 This paper is devoted to study decay rates for the system \eqref{MHD} and proves the optimal convergence rates of its solutions under suitable assumptions. Bian-Guo \cite{bian} has obtained the global existence of smooth solutions to the system \eqref{MHD} under the condition that
 the initial data are close to the equilibrium state in $H^3$-framework.
More precisely, this result is expressed in the following.

\begin{prop}\label{thm1}
Assume that initial data are close enough to the constant state
$(\bar{\rho},0,0,\bar{k},0)$, i.e. there exists a constant $\delta_0$ such that
if
\begin{equation} \|(\rho_0-\bar{\rho}, u_{0}, h_0, k_0-\bar{k}, \varepsilon_0)\|_{H^3(\mathbb{R}^3)}\leq \delta_0,
 \end{equation}
 then the system
(\ref{MHD}) admits a unique smooth solution $(\rho,u,h,k,\varepsilon)$ such that for any
$t\in [0,\infty)$,
\begin{equation}\label{propc1.7}
\begin{split}
&\|(\rho-\bar{\rho}, u, h, k-\bar{k}, \varepsilon)\|_{H^{3}}^2
+\int_0^t\|\nabla \rho\|_{H^2}^2+\|(\nabla u, \nabla h,\nabla k, \nabla \varepsilon)\|_{H^3}^2\mbox{d}s\\
&\leq
C\|(\rho_0-\bar{\rho},u_0, h_0, k_0-\bar{k}, \varepsilon_0)\|_{H^{3}}^2,
\end{split}
\end{equation}
where $C$ is a positive constant.
\end{prop}

Based on this stability result, the main purpose in this paper is to investigate the optimal convergence rates in time to the stationary solution. We remark that the convergence rate is an important topic in the study of the fluid dynamics for the purpose of the computation \cite{ene, goubet}.  The main idea in this paper is to combine the $L^p$-$L^q$ estimates for the linearized equations and an improved energy method which includes the estimation on the higher power of $L^2$-norm of solutions. By doing this, the optimal convergence rates for the solutions to the nonlinear problem \eqref{MHD} in various norms can be obtained and are stated in the following theorem.


\begin{thm}\label{thm2}
Let $\delta_0$ be the constant defined in Proposition \ref{thm1}. There exist constants $\delta_1\in(0,\delta_0)$ and $C>0$ such that the following holds. For any $\delta\leq\delta_1$, if
\begin{equation} \|(\rho_0-\bar{\rho}, u_{0}, h_0, k_0-\bar{k}, \varepsilon_0)\|_{H^3(\mathbb{R}^3)}\leq \delta,
 \end{equation}
 and for some $p\in[1, \frac{6}{5})$,
 \begin{align}\label{1.9}
 \rho_0-\bar{\rho}, u_{0}, h_0, k_0-\bar{k}, \varepsilon_0\in L^p(\mathbb{R}^3),
 \end{align}
 then the smooth solution $(\rho,u,h,k,\varepsilon)$ in Proposition \ref{thm1} enjoys the estimates for
$t\in [0,\infty)$,
\begin{equation}\label{decay1}
\|(\rho-\bar{\rho}, u, h, k-\bar{k}, \varepsilon)\|_{q}
\leq
C(1+t)^{-\sigma(p,q;0)},\ \ 2\leq q\leq 6,
\end{equation}
\begin{equation}\label{decay2}
\|(\rho-\bar{\rho}, u, h, k-\bar{k}, \varepsilon)\|_{\infty}
\leq
C(1+t)^{-\sigma(p,2;1)},
\end{equation}
\begin{equation}\label{decay3}
\|\nabla(\rho-\bar{\rho}, u, h, k-\bar{k}, \varepsilon)\|_{H^2}
\leq
C(1+t)^{-\sigma(p,2;1)},
\end{equation}
\begin{equation}\label{decay4}
\|(\rho_t, u_t, h_t, k_t, \varepsilon_t)\|_{2}
\leq
C(1+t)^{-\sigma(p,2;1)},
\end{equation}
where $\sigma(p,q;l)$ are defined by
\begin{align}\label{sigmadefinition}
\sigma(p,q;l)=\frac{3}{2}(\frac{1}{p}-\frac{1}{q})+\frac{l}{2}.
\end{align}
\end{thm}

\begin{rem}
\eqref{1.9} shows that the perturbation of initial data around the constant state $(\bar{\rho},0,0,\bar{k}, 0)$  is bounded in $L^p$-norm, for some $p\in[1,\frac{6}{5})$, which need not be small.
\end{rem}
\begin{rem}
The linearized equations of \eqref{MHD} around the constant state  $(\bar{\rho},0,0,\bar{k}, 0)$  take the following form:
\begin{align}\label{ke}
\begin{cases}
a_t+ \gamma\mbox{div}v=0,\\
v_t+\gamma\nabla a-\lambda\Delta v-\lambda\nabla\mbox{div}v=0,\\
 h_t-\lambda\Delta h=0,\\
m_t-\lambda\Delta m=0,\\
 \varepsilon_t-\lambda\Delta\varepsilon
=0,
\end{cases}
\end{align}
where $\gamma$, $\lambda$ are positive constants which will be given precisely in Section 2. Compared to the decay estimates of the solutions to the above linearized equations by using Fourier analysis \cite{kobayashi} stated in Lemma 2.1 in the next section, Theorem \ref{thm2} gives the optimal decay rates for the solution in $L^q$-norm, for any $2\leq q\leq 6$, and its first order estimates in $L^2$-norm. Note that the convergence rates of  the derivatives of higher order in $L^2$-norm and the solution in $L^{\infty}$-norm are not the same as those for linearized equations.
\end{rem}
\begin{rem}
We mainly use the method in \cite{tong} to prove Theorem \ref{thm2}. But our problem is much more difficult because of the strong coupling between velocity, total enthalpy, turbulent kinetic energy and rate of viscous dissipation. Moreover, our result is better than that in \cite{tong}. Here we assume $L^p$-norm of initial data $\rho_0-\bar{\rho}, u_{0}, h_0, k_0-\bar{k}, \varepsilon_0$ is bounded, for some $p\in[1,\frac65)$, instead of $\|(\rho_0-\bar{\rho}, u_{0}, h_0, k_0-\bar{k}, \varepsilon_0)\|_{L^1}<\infty$.
\end{rem}

 \textbf{Notation:} Throughout the paper, $C$ stands for a
general constant, and may change from line to line.
The norm $\|(A,B)\|_{X}$ is equivalent to $\|A\|_{X}+\|B\|_{X}$ and $\|A\|_{X\cap Y}=\|A\|_{X}+\|A\|_{Y}$.
The norms in the Sobolev Spaces $H^m(\mathbb{R}^3)$ and $W^{m,q}(\mathbb{R}^3)$ are denoted respectively by $\|\cdot\|_{H^m}$ and $\|\cdot\|_{W^{m,q}}$ for $m\geq0$, $q\geq 1$.  In particular, for $m=0$, we will simply use $\|\cdot\|_p$. Moreover,
 $\langle\cdot,\cdot\rangle$ denotes the inner-product in $L^2(\mathbb{R}^3)$.
 Finally, $$\nabla=(\partial_1,\partial_2, \partial_3),\  \partial_i=\partial_{x_i},\  i=1,\ 2,\ 3,$$
 and for any integer $l\geq0$, $\nabla^lf$ denotes all derivatives up to $l$-order of the function $f$. And for multi-indices $\alpha$, $\beta$ and $\xi$ $$\alpha=(\alpha_1,\alpha_2,\alpha_3),\ \beta=(\beta_1,\beta_2,\beta_3),\ \xi=(\xi_1,\xi_2,\xi_3),$$
we use $$\partial_x^{\alpha}=\partial_{x_1}^{\alpha_1}\partial_{x_2}^{\alpha_2}\partial_{x_3}^{\alpha_3},\ |\alpha|=\sum_{i=1}^{3}\alpha_i,$$
and $C_{\alpha}^{\beta}=\frac{\alpha!}{\beta!(\alpha-\beta)!}$ when $\beta\leq \alpha$.

\section{Preliminaries}
We will reformulate the problem \eqref{MHD} as follows. Set
$$\gamma=\sqrt{p'(\bar{\rho})+\bar{k}},\ \   \lambda=\frac{1}{\bar{\rho}}.$$ Introducing new variables by $$a=\rho-\bar{\rho},\  v=\frac{1}{\gamma\lambda},\  h=h, \ m=k-\bar{k},\  \varepsilon=\varepsilon,$$ the initial value problem \eqref{MHD} is reformulated as
\begin{align}\label{ke-new}
\begin{cases}
a_t+ \gamma\mbox{div}v=F_1,\\
v_t+\gamma\nabla a-\lambda\Delta v-\lambda\nabla\mbox{div}v=F_2,\\
 h_t-\lambda\Delta h=F_3,\\
m_t-\lambda\Delta m=F_4,\\
 \varepsilon_t-\lambda\Delta\varepsilon
=F_5,\\
(a, v, h, m,\varepsilon)(x,0)=(a_0, v_0, h_0, m_0, \varepsilon_0)\rightarrow0 \ \mbox{as} \ |x|\rightarrow\infty,
\end{cases}
\end{align}
where
\begin{align}
\begin{split}
&F_1=-\gamma\lambda\mbox{div}(au),\\
&F_2=(\frac{1}{\rho}-\frac{1}{\bar{\rho}})(\Delta v+\nabla\mbox{div}v)-\frac{1}{\gamma\lambda}\big(\frac{p'(a+\bar{\rho})}{a+\bar{\rho}}
-\frac{p'(\bar{\rho})}{\bar{\rho}}+\frac{2(m+\bar{k})}{3(a+\bar{\rho})}-\frac{2\bar{k}}{3\bar{\rho}}\big)\nabla a\\
&-\frac{2}{3\gamma\lambda}\nabla m,\\
&F_3=(\frac{1}{\rho}-\frac{1}{\bar{\rho}})\Delta h-\gamma \lambda p'(a+\bar{\rho})\mbox{div}v+\frac{1}{a+\bar{\rho}}S^1_k -\gamma\lambda v\cdot\nabla h,\\
&F_4=(\frac{1}{\rho}-\frac{1}{\bar{\rho}})\Delta m+\frac{1}{a+\bar{\rho}}G^1-\varepsilon -\gamma\lambda v\cdot\nabla m,\\
&F_5=(\frac{1}{\rho}-\frac{1}{\bar{\rho}})\Delta \varepsilon+\frac{C_1G^1\varepsilon}{(a+\bar{\rho})(m+\bar{k})}-\frac{C_2\varepsilon^2}{m+\bar{k}}S_k -\gamma\lambda v\cdot\nabla \varepsilon,
\end{split}
\end{align}
 with $S^1_k$ and $G^1$ the corresponding $S_k$ and $G$ in the variables of $(a, v, h, m, \varepsilon)$.

Let $U=(a, v)$, $U_0=(a_0, v_0)$, $F=(F_1,F_2)$,
\begin{displaymath}
\mathrm{A}=
\left( \begin{array}{ccc}
0&  \gamma \mbox{div}\\
\gamma\nabla& \lambda\Delta+\lambda\nabla\mbox{div}
\end{array}
\right)
\end{displaymath}
and
$E(t)$ be the semigroup generated by the linear operator $A$, then we can rewrite the solution for the first two equations of the nonlinear problem \eqref{MHD} as
\begin{align}\label{integral1}
U(t)=E(t)U_0+\int_0^tE(t-s)F\mbox{d}s.
\end{align}
The semigroup $E(t)$ has the following properties on the decay in time, which can be found in \cite{kobayashi1,kobayashi} and will be applied to the integral formula \eqref{integral1}.
\begin{lem}\label{lem1}
Let $l\geq 0$ be an integer and $1\leq p\leq 2\leq q<\infty$. Then for any $t\geq 0$, it holds that
\begin{align}\label{decaylemma1}
\|\nabla^lE(t)U_0\|_q\leq C(1+t)^{-\sigma(p,q;l)}\|U_0\|_{L^p\cap H^l},
\end{align}
with $\sigma(p,q;l)$ defined by \eqref{sigmadefinition}.
\end{lem}

To treat the last three equations, we introduce the  semigroup $S(t)$ generated by $\lambda \Delta$, then $(\ref{ke-new})_3$-$(\ref{ke-new})_5$ become
\begin{align}\label{integral2}
\begin{split}
&h(t)=S(t)h_0+\int_0^tS(t-s)F_3\mbox{d}s,\\
&m(t)=S(t)m_0+\int_0^tS(t-s)F_4\mbox{d}s,\\
&\varepsilon(t)=S(t)\varepsilon_0+\int_0^tS(t-s)F_5\mbox{d}s.
\end{split}
\end{align}
We state the large-time behavior of solutions to the last three equations of the system \eqref{ke-new} as the following lemma which can be obtained by direct calculation or can refer to \cite{zuazua}.
\begin{lem}\label{lem2}
For the solution $(h, m, \varepsilon)$ of the last three equations of the system \eqref{ke-new} with Cauchy data $h(x,0)=h_0$, $m(x,0)=m_0$, $\varepsilon(x,0)=\varepsilon_0$, there exists a constant $C$ such that
\begin{align}\label{decaylemma2}
\begin{split}
&\|(\nabla^lh,\nabla^lm, \nabla^l\varepsilon) \|_q\leq C(1+t)^{-\sigma(p,q;l)}\|(h_0,m_0,\varepsilon_0)\|_{p}\\
&+C\int_0^t(1+t-s)^{-\sigma(p,q;l)}\|(F_3,F_4,F_5)\|_p,\ l=0,\ 1,
\end{split}
\end{align}
for any $t\geq0$, $1\leq p,q\leq+\infty$, as well as $\sigma$ is defined by \eqref{sigmadefinition}.
\end{lem}

For later use we list some Sobolev inequalities as follows, cf. \cite{adams, deckelnick}.
\begin{lem}\label{2.2}
Let $\Omega\subset\mathbb{R}^3$ be the whole space $\mathbb{R}^3$, or half space $\mathbb{R}_+$ or the exterior domain of a bounded region with smooth boundary. Then

(i) $\|f\|_{L^6(\Omega)}\leq C\|\nabla f\|_{L^2(\Omega)}$, for $f\in H^1(\Omega)$.

(ii) $\|f\|_{L^p(\Omega)}\leq C\|f\|_{H^1(\Omega)}$, for $2\leq p\leq 6$.

(iii) $\|f\|_{C^0(\bar{\Omega})}\leq C\|f\|_{W^{1,p}(\Omega)}\leq C\|\nabla f\|_{H^1(\Omega)}$, for $f\in H^2(\Omega)$.

(iv) $|\int_{\Omega}f\cdot g\cdot h\mbox{d}x|\leq \varepsilon\|\nabla f\|_{L^2}^2+\frac{C}{\varepsilon}\|g\|_{H^1}^2\|h\|_{L^2}^2$, for $\varepsilon>0$, $f,\  g \in H^1(\Omega)$,  $h\in L^2(\Omega)$.

(v) $|\int_{\Omega}f\cdot g\cdot h\mbox{d}x|\leq \varepsilon\|g\|_{L^2}^2+\frac{C}{\varepsilon}\|\nabla f\|_{H^1}^2\|h\|_{L^2}^2$, for $\varepsilon>0$, $f\in H^2(\Omega)$,  $g,\ h\in L^2(\Omega)$.
\end{lem}
Finally, the following elementary inequality \cite{tong} will also be used.
\begin{lem}\label{lem2.5}
If $r_1>1$ and $r_2\in[0,r_1]$, then it holds that
\begin{align}
\int_0^t(1+t-s)^{-r_1}(1+s)^{-r_2}\mbox{d}s\leq C_1(r_1,r_2)(1+t)^{-r_2},
\end{align}
where $C_1(r_1,r_2)$ is defined by
\begin{align}\label{2.10}
C_1(r_1,r_2)=\frac{2^{r_2+1}}{r_1-1}.
\end{align}
\end{lem}

\section{Basic estimates}
In this section we shall establish two basic inequalities for the proof of the optimal convergence rates in section 4. One inequality is the decay rate of the first order derivatives, while the other is the energy estimate.
\begin{lem}\label{lem3.5}
Let $W=(U, h, m, \varepsilon)$ be the solution to the problem \eqref{ke-new}, then under the assumptions of Theorem \ref{thm2}, we have
\begin{align}\label{inequality3.5}
\begin{split}
&\|\nabla W \|_2\leq CE_0(1+t)^{-\sigma(p,2;1)}
+C\int_0^t(1+t-s)^{-\sigma(p,2;1)}\|\nabla W\|_{H^2}\mbox{d}s,
\end{split}
\end{align}
with $E_0=\|W(0)\|_{L^p\cap H^1}=\|(U_0, h_0, m_0, \varepsilon_0)\|_{L^p\cap H^1}=\|(a_0, v_0, h_0, m_0, \varepsilon_0)\|_{L^p\cap H^1}$ and $1\leq p<\frac65$.
\end{lem}
\begin{proof}
Let $l=1$ in \eqref{decaylemma1}, we then have from \eqref{integral1}
\begin{align*}
\begin{split}
&\|\nabla U(t) \|_2\leq CE_0(1+t)^{-\sigma(p,2;1)}
+C\int_0^t(1+t-s)^{-\sigma(p,2;1)}\|(F_1,F_2)\|_{L^p\cap H^1}\mbox{d}s,
\end{split}
\end{align*}
which together with \eqref{decaylemma2} implies that
\begin{align}\label{f1f1f3f4f5}
\begin{split}
&\|\nabla W(t)\|_2\leq CE_0(1+t)^{-\sigma(p,2;1)}
+C\int_0^t(1+t-s)^{-\sigma(p,2;1)}[\|(F_1,F_2)\|_{L^p\cap H^1}\\
&+\|(F_3,F_4,F_5)\|_{p}]\mbox{d}s.
\end{split}
\end{align}
For $\|(F_1,F_2)\|_{L^p\cap H^1}$, we estimate as follows. For $1\leq p < \frac65$, the term in $F_1$ can be estimated as
\begin{align*}
\|-\gamma\lambda\mbox{div}(av)\|_p&\leq C(\|\partial_iav^i\|_p+\|a\partial_i v^i\|_p)\\
&\leq C\|v\|_{\frac{2p}{2-p}}\|\nabla a\|_2+C\|\nabla v\|_{2}\|a\|_{\frac{2p}{2-p}}\\
&\leq C\delta\|(\nabla a, \nabla v)\|_2.
\end{align*}
With the help of H\"{o}lder inequality and Lemma \ref{2.2}, it holds that
\begin{align*}
\|-\gamma\lambda\mbox{div}(av)\|_{H^1}&\leq C(\|\partial_iav^i\|_2+\|a\partial_i v^i\|_2+\|\nabla\mbox{div}(av)\|_2)\\
&\leq C\delta\|(\nabla a, \nabla v)\|_{H^1}.
\end{align*}
Similarly, the terms in $F_2$ can be estimated as
\begin{align*}
\|(\frac{1}{\rho}-\frac{1}{\bar{\rho}})(\Delta v+\nabla\mbox{div}v)\|_p\leq C\|\nabla^2 v\|_2\|a\|_{\frac{2p}{2-p}}\leq C\delta\|\nabla^2 v\|_2,
\end{align*}
\begin{align*}
&\|(\frac{1}{\rho}-\frac{1}{\bar{\rho}})(\Delta v+\nabla\mbox{div}v)\|_{H^1}\\
&\leq C\|(\frac{1}{\rho}-\frac{1}{\bar{\rho}})(\Delta v+\nabla\mbox{div}v)\|_2+\|\nabla((\frac{1}{\rho}-\frac{1}{\bar{\rho}})(\Delta v+\nabla\mbox{div}v))\|_2\\
&\leq C(\|a\|_{\infty}\|\nabla^2 v\|_2+\|a\|_{\infty}\|\nabla^3 v\|_2+\|\nabla a\|_{\infty}\|\nabla^2 v\|_2)\\
&\leq C\delta\|\nabla v\|_{H^2},
\end{align*}
\begin{align*}
&\|-\frac{1}{\gamma\lambda}\big(\frac{p'(a+\bar{\rho})}{a+\bar{\rho}}
-\frac{p'(\bar{\rho})}{\bar{\rho}}+\frac{2(m+\bar{k})}{3(a+\bar{\rho})}-\frac{2\bar{k}}{3\bar{\rho}}\big)\nabla a\|_p\\
&\leq C(\|\nabla a\|_2\|a\|_{\frac{2p}{2-p}}+\|\nabla a\|_2\|m\|_{\frac{2p}{2-p}})\\
&\leq C\delta\|\nabla a\|_2,
\end{align*}
\begin{align*}
&\|-\frac{1}{\gamma\lambda}\big(\frac{p'(a+\bar{\rho})}{a+\bar{\rho}}
-\frac{p'(\bar{\rho})}{\bar{\rho}}+\frac{2(m+\bar{k})}{3(a+\bar{\rho})}-\frac{2\bar{k}}{3\bar{\rho}}\big)\nabla a\|_{H^1}\\
&\leq C\big(\|\big(\frac{p'(a+\bar{\rho})}{a+\bar{\rho}}
-\frac{p'(\bar{\rho})}{\bar{\rho}}+\frac{2(m+\bar{k})}{3(a+\bar{\rho})}-\frac{2\bar{k}}{3\bar{\rho}}\big)\nabla a\|_2\\
&+\|\nabla\big[\big(\frac{p'(a+\bar{\rho})}{a+\bar{\rho}}
-\frac{p'(\bar{\rho})}{\bar{\rho}}+\frac{2(m+\bar{k})}{3(a+\bar{\rho})}-\frac{2\bar{k}}{3\bar{\rho}}\big)\nabla a\big]\|_2\big)\\
&\leq C(\|a\nabla a+m\nabla a\|_2+\|a\nabla^2 a+m\nabla^2 a+(\nabla a)^2+\nabla a\nabla m\|_2)\\
&\leq C\delta \|\nabla a\|_{H^2},
\end{align*}
\begin{align*}
\|-\frac{2}{3\gamma\lambda}\nabla m\|_p\leq \|\nabla m\|_2\|m+\bar{k}\|_6\|\frac{1}{m+\bar{k}}\|_{\frac{3p}{3-2p}}
\leq C\delta \|\nabla m\|_2,
\end{align*}
\begin{align*}
&\|-\frac{2}{3\gamma\lambda}\nabla m\|_{H^1}\leq C(\|\nabla m\|_2+\|\nabla^2m\|_2)
\\
&\leq C(\|\nabla m\|_3\|m+\bar{k}\|_6\|\frac{1}{m+\bar{k}}\|_{\infty}+\|\nabla^2 m\|_3\|m+\bar{k}\|_6\|\frac{1}{m+\bar{k}}\|_{\infty})\\
&\leq C\delta\|\nabla m\|_{H^2}.
\end{align*}
Thus we have
\begin{align}\label{f1f2}
\|(F_1,F_2)\|_{L^p\cap H^1}\leq C\delta \|(\nabla a,\nabla v, \nabla m)\|_{H^2}.
\end{align}
Next, we estimate for $\|(F_3,F_4,F_5)\|_{p}$. Almost as same as the estimation of $\|(F_1,F_2)\|_{L^p\cap H^1}$, we get
\begin{align}\label{f3f4f5}
\|(F_3,F_4,F_5)\|_{p}\leq C\delta \|(\nabla a,\nabla v, \nabla h, \nabla m, \nabla \varepsilon)\|_{H^2}.
\end{align}
Inserting \eqref{f1f2} and \eqref{f3f4f5} into \eqref{f1f1f3f4f5}, we complete the proof of Lemma \ref{lem3.5}.
\end{proof}
\begin{lem}\label{lem3.6}
Let $W=(U, h, m, \varepsilon)$ be the solution to the problem \eqref{ke-new}, then under the assumptions of Theorem \ref{thm2}, if $\delta>0$ is sufficiently small then it holds that
\begin{align}\label{inequality3.6}
\frac{\mbox{d}M(t)}{\mbox{d}t}+\|\nabla^2a\|_{H^1}^2+\|\nabla^2(v,h,m,\varepsilon)\|_{H^2}^2\leq C\delta \|\nabla(a, v, h, m, \varepsilon)\|_{2}^2,
\end{align}
where $M(t)$ is equivalent to $\|\nabla(a, v,h,m,\varepsilon)\|_{H^2}^2$, i.e., there exists a positive constant $C_2$ such that
\begin{align}\label{energy}
C_2^{-1}\|\nabla(a, v, h, m, \varepsilon)\|_{H^2}^2\leq M(t)\leq C_2\|\nabla(a, v, h, m, \varepsilon)\|_{H^2}^2.
\end{align}
\end{lem}
\begin{proof}
Let $\alpha$ be any multi-index with $1\leq |\alpha|\leq 3$. Applying the operator $\partial_x^{\alpha}$ to \eqref{ke-new} and then taking inner product with $\partial_x^{\alpha}a$, $\partial_x^{\alpha}v$, $\partial_x^{\alpha}h$, $\partial_x^{\alpha}m$ and $\partial_x^{\alpha}\varepsilon$, one gets
\begin{align*}
&\frac{1}{2}\frac{\mbox{d}}{\mbox{d}t}\|\partial_x^{\alpha}(a,v,h,m,\varepsilon)\|_2^2
+\lambda\|\nabla\partial_x^{\alpha}v\|_2^2+\lambda\|\mbox{div}\partial_x^{\alpha}v\|_2^2+\lambda\|\nabla\partial_x^{\alpha}(h,m,\varepsilon)\|_2^2\\
&=\langle\partial_x^{\alpha}a,\partial_x^{\alpha}F_1\rangle+\langle\partial_x^{\alpha}v,\partial_x^{\alpha}F_2\rangle
+\langle\partial_x^{\alpha}h,\partial_x^{\alpha}F_3\rangle+\langle\partial_x^{\alpha}m,\partial_x^{\alpha}F_4\rangle
\\
&+\langle\partial_x^{\alpha}\varepsilon,\partial_x^{\alpha}F_5\rangle\\
&=:J_1^{\alpha}(t)+J_2^{\alpha}(t)+J_3^{\alpha}(t)+J_4^{\alpha}(t)+J_5^{\alpha}(t),
\end{align*}
where $J_i^{\alpha}(t)$, $i=1,2,3,4,5$, are the corresponding terms in the above equation which will be estimated as follows.

Now let's estimate for $J_1^{\alpha}(t)$. It follows from Lemma \ref{2.2} that
\begin{align*}
J_1^1&\sim\langle\partial_j\partial_i(av^i),\partial_ja\rangle\sim\langle\partial_j\partial_iav^i+\partial_ia\partial_jv^i
+\partial_ja\partial_iv^i+\partial_i\partial_jv^ia,\partial_ja\rangle\\
&\leq C\|\nabla^2a\|_2\|\nabla a\|_2\|v\|_{\infty}+\|\nabla a\|_2^2\|\nabla v\|_{\infty}+\|\nabla^2v\|_2\|\nabla a\|_2\|a\|_{\infty}\\
&\leq \delta\|\nabla a\|_2^2+\frac{C}{\delta}\|\nabla^2a\|_2^2\|v\|_{H^2}^2+\frac{C}{\delta}\|\nabla^2v\|_2^2\|a\|_{H^2}^2,
\end{align*}
\begin{align*}
J_1^2&\sim\langle\partial_x^{\beta}\partial_i(av^i),\partial_x^{\alpha+\xi}a\rangle\sim\langle\nabla^2av+\nabla a\nabla v+\nabla^2va,\nabla^3a\rangle\\
&\leq \delta\|\nabla^3 a\|_2^2+\frac{C}{\delta}(\|\nabla^2a\|_2^2\|v\|_{\infty}^2+\|\nabla a\|_2^2\|\nabla v\|_{\infty}^2+\|\nabla^2v\|_2^2\|a\|_{\infty}^2)\\
&\leq \delta\|\nabla^3 a\|_2^2+\frac{C}{\delta}(\|\nabla^2a\|_2^2\|v\|_{H^2}^2+\|\nabla a\|_2^2\|\nabla v\|_{H^2}^2+\|\nabla^2v\|_2^2\|a\|_{H^2}^2),
\end{align*}
where $\alpha=\beta+\xi$, with $|\beta|=|\xi|=1$.

For $J_1^3$, we have from multi-indices $\beta$ and $\xi$ with $|\beta|=2$, $|\xi|=1$,
\begin{align*}
J_1^3&\sim\langle\partial_x^{\beta+\xi}\partial_i(av^i),\partial_x^{\beta+\xi}a\rangle
\\
&\sim -\int_{\mathbb{R}^3}(\partial_x^{\beta+\xi}a)^2\mbox{div}v\mbox{d}x+\langle\nabla a\nabla^3v+\nabla^3 a\nabla v+\nabla^2a\nabla^2v+a\nabla^4v,\nabla^3a\rangle\\
&\leq \delta\|\nabla^3 a\|_2^2+\frac{C}{\delta}(\|\nabla^2a\|_3^2\|\nabla^2v\|_{6}^2+\|\nabla^3 v\|_2^2\|\nabla a\|_{\infty}^2+\|\nabla^4v\|_2^2\|a\|_{\infty}^2)\\
&\leq \delta\|\nabla^3 a\|_2^2+\frac{C}{\delta}(\|\nabla^2a\|_2^2\|v\|_{H^2}^2+\|\nabla a\|_2^2\|\nabla v\|_{H^2}^2+\|\nabla^2v\|_2^2\|a\|_{H^2}^2).
\end{align*}
Hence
\begin{align}\label{j1}
J_1^{\alpha}\leq C\delta\sum_{1\leq |\alpha|\leq3}\|\partial_x^{\alpha}a\|_2^2+C\delta\sum_{1\leq |\alpha|\leq4}\|\partial_x^{\alpha}v\|_2^2.
\end{align}
Similarly, by Lemma \ref{2.2} and H\"{o}lder inequality,  $J_2^{\alpha}$ can be estimated as
\begin{align*}
&J_2^1\sim\langle\partial_j[(\frac{1}{\rho}-\frac{1}{\bar{\rho}})(\Delta v+\nabla\mbox{div}v)-\frac{1}{\gamma\lambda}\big(\frac{p'(a+\bar{\rho})}{a+\bar{\rho}}
-\frac{p'(\bar{\rho})}{\bar{\rho}}+\frac{2(m+\bar{k})}{3(a+\bar{\rho})}-\frac{2\bar{k}}{3\bar{\rho}}\big)\nabla a\\
&-\frac{2}{3\gamma\lambda}\nabla m],\partial_jv\rangle\\
&\sim\langle\partial_j a(\Delta v+\nabla\mbox{div}v)+a(\partial_j\Delta v+\partial_j\nabla\mbox{div}v)+\partial_ja\cdot\nabla a+\partial_jm\cdot\nabla a+a\partial_j\nabla a\\
&+m\partial_j\nabla a+\partial_j\nabla m,\partial_jv\rangle\\
&\leq C(\|\nabla a\|_3\|\nabla v\|_6\|\nabla^2 v\|_2+\|\nabla^3 v\|_2\|\nabla v\|_2\|a\|_{\infty}+\|\nabla (a,m)\|_3\|\nabla a\|_6\|\nabla v\|_2\\
&+\|\nabla^2 a\|_2\|\nabla v\|_2\|m\|_{\infty}+\|\nabla^2 a\|_2\|\nabla v\|_2\|a\|_{\infty}+\|\nabla^2 m\|_2\|\nabla v\|_3\|\nabla m\|_2)\\
&\leq \delta (\|\nabla^2v\|_2^2+\|\nabla^3v\|_2^2+\|\nabla^2(a,m)\|_2^2)+\frac{C}{\delta}(\|\nabla v\|_2^2\|(a,m)\|_{H^2}^2+\|\nabla v\|_{H^1}^2\|\nabla m\|_2^2)\\
&\leq C\delta\|\nabla^2(a,m)\|_2^2+C\delta\sum_{1\leq|\alpha|\leq3}\|\partial_x^{\alpha}v\|_2^2,
\end{align*}
\begin{align*}
&J_2^2\sim\langle\partial_x^{\beta}[(\frac{1}{\rho}-\frac{1}{\bar{\rho}})(\Delta v+\nabla\mbox{div}v)-\frac{1}{\gamma\lambda}\big(\frac{p'(a+\bar{\rho})}{a+\bar{\rho}}
-\frac{p'(\bar{\rho})}{\bar{\rho}}+\frac{2(m+\bar{k})}{3(a+\bar{\rho})}-\frac{2\bar{k}}{3\bar{\rho}}\big)\nabla a\\
&-\frac{2}{3\gamma\lambda}\nabla m],\partial_x^{\alpha+\xi}v\rangle\\
&\sim\langle\nabla a(\Delta v+\nabla\mbox{div}v)+a\nabla^3v+\nabla a\nabla a+\nabla a\nabla m+a\nabla^2 a+m\nabla^2a+\nabla^2 m,\nabla^3v\rangle\\
&\leq C(\|\nabla a\|_3\|\nabla^2 v\|_6\|\nabla^3 v\|_2+\|\nabla^3 v\|_2^2\|a\|_{\infty}+\|\nabla a\|_3\|\nabla a\|_6\|\nabla^3 v\|_2\\
&+\|\nabla a\|_6\|\nabla^3 v\|_2\|\nabla m\|_3+\|\nabla^2 a\|_2\|\nabla^3 v\|_2\|(a,m)\|_{\infty}+\|\nabla^2 m\|_3\|\nabla^2 v\|_2\| m+\bar{k}\|_6)\\
&\leq C\delta\sum_{1\leq|\alpha|\leq3}\|\partial_x^{\alpha}a\|_2^2
+C\delta\sum_{1\leq|\alpha|\leq3}\|\partial_x^{\alpha}(v,m)\|_2^2,\ \mbox{with}\  \alpha=\beta+\xi,\ |\beta|=|\xi|=1,
\end{align*}
\begin{align*}
&J_2^3\sim\langle\partial_x^{\beta}[(\frac{1}{\rho}-\frac{1}{\bar{\rho}})(\Delta v+\nabla\mbox{div}v)-\frac{1}{\gamma\lambda}\big(\frac{p'(a+\bar{\rho})}{a+\bar{\rho}}
-\frac{p'(\bar{\rho})}{\bar{\rho}}+\frac{2(m+\bar{k})}{3(a+\bar{\rho})}-\frac{2\bar{k}}{3\bar{\rho}}\big)\nabla a\\
&-\frac{2}{3\gamma\lambda}\nabla m],\partial_x^{\alpha+\xi}v\rangle\\
&\sim\langle a\nabla^4v+\nabla a\nabla^3 v+\nabla^2 a\nabla^2 v+a\nabla^3 a+\nabla a\nabla^2a+\nabla m\nabla^2 a+\nabla^2m\nabla a\\
&+m\nabla^3a+\nabla^3 m,\nabla^4v\rangle\\
&\leq C\delta\|\nabla^4v\|_2^2+\frac{C}{\delta}(\|\nabla a\|_{\infty}^2\|\nabla^3 v\|_2^2+\|\nabla^2 v\|_3^2\|\nabla^2a\|_6^2+\|\nabla^3 a\|_2^2\| a\|_{\infty}^2\\
&+\|\nabla^2 a\|_2^2\|\nabla(a,m)\|_{\infty}^2+\|\nabla^2 m\|_2^2\|\nabla a\|_{\infty}^2+\| \nabla^3a\|_2^2\|m\|_{\infty}^2+\|\nabla^3m\|_3\|\nabla m\|_2\\
&\leq C\delta\sum_{1\leq|\alpha|\leq3}\|\partial_x^{\alpha}a\|_2^2
+C\delta\sum_{1\leq|\alpha|\leq4}\|\partial_x^{\alpha}(v,m)\|_2^2,\ \mbox{with}\  \alpha=\beta+\xi,\ |\beta|=2,\ |\xi|=1.
\end{align*}
Incorporating the above estimates, it holds that
\begin{align}\label{j2}
J_2^{\alpha}(t)\leq C\delta\sum_{1\leq|\alpha|\leq3}\|\partial_x^{\alpha}a\|_2^2
+C\delta\sum_{1\leq|\alpha|\leq4}\|\partial_x^{\alpha}(v,m)\|_2^2.
\end{align}

Moreover, $J_3^{\alpha}(t)$,  $J_4^{\alpha}(t)$ and $J_5^{\alpha}(t)$ can be estimated similarly by using Lemma \ref{2.2} and H\"{o}lder inequality,
\begin{align}\label{j3}
|J_3^{\alpha}(t),J_4^{\alpha}(t), J_5^{\alpha}(t) |\leq C\delta\sum_{1\leq|\alpha|\leq3}\|\partial_x^{\alpha}a\|_2^2
+C\delta\sum_{1\leq|\alpha|\leq4}\|\partial_x^{\alpha}(v,h,m,\varepsilon)\|_2^2.
\end{align}

On the other hand, we apply $\partial_x^{\alpha}$ to $\eqref{ke-new}_2$ with $1\leq|\alpha|\leq2$, and then take inner product with $\nabla\partial_x^{\alpha}a$ to yield
\begin{align}\label{ke-new-1}
\langle\partial_x^{\alpha}v_t,\nabla\partial_x^{\alpha}a\rangle+\gamma\|\nabla\partial_x^{\alpha}a\|_2^2
=\lambda\langle\partial_x^{\alpha}\Delta v+\partial_x^{\alpha}\nabla\mbox{div}v,\partial_x^{\alpha}\nabla a\rangle+\langle\partial_x^{\alpha}F_2, \partial_x^{\alpha}\nabla a\rangle,
\end{align}
and similarly from $\eqref{ke-new}_1$,
\begin{align}\label{ke-new-2}
\langle\partial_x^{\alpha}v,\nabla\partial_x^{\alpha}a_t\rangle=-\gamma\langle\partial_x^{\alpha} v,\nabla\partial_x^{\alpha} \mbox{div}v\rangle+\langle\partial_x^{\alpha}v, \nabla\partial_x^{\alpha}F_1\rangle.
\end{align}
Adding \eqref{ke-new-1} and \eqref{ke-new-2} together implies
\begin{align}
\begin{split}
&\frac{\mbox{d}}{\mbox{d}t}\langle\partial_x^{\alpha}v,\nabla\partial_x^{\alpha}a\rangle
+\gamma\|\nabla\partial_x^{\alpha}a\|_2^2=\lambda\langle\partial_x^{\alpha}\Delta v+\partial_x^{\alpha}\nabla\mbox{div}v,\partial_x^{\alpha}\nabla a\rangle\\
&+\langle\partial_x^{\alpha}F_2, \partial_x^{\alpha}\nabla a\rangle-\gamma\langle\partial_x^{\alpha} v,\nabla\partial_x^{\alpha} \mbox{div}v\rangle+\langle\partial_x^{\alpha}v, \nabla\partial_x^{\alpha}F_1\rangle.
\end{split}
\end{align}
By H\"{o}lder inequality and similar to the estimation of $J_2^{\alpha}(t)$, the right hand side can be bounded by
\begin{align}\label{j4}
\frac{\gamma}{2}\|\nabla\partial_x^{\alpha}a\|_2^2+C\sum_{1\leq|\alpha|\leq2}\|\partial_x^{\alpha}\nabla v\|_{H^1}^2+C\delta\sum_{1\leq|\alpha|\leq3}\|\partial_x^{\alpha}a\|_2^2
+C\delta\sum_{1\leq|\alpha|\leq4}\|\partial_x^{\alpha}(v,h,m,\varepsilon)\|_2^2.
\end{align}
Therefore, if we define
\begin{align*}
M(t)=C_1\sum_{1\leq|\alpha|\leq3}\|\partial_x^{\alpha}(a,v,h,m,\varepsilon)\|_2^2
+\sum_{1\leq|\alpha|\leq2}\langle\partial_x^{\alpha}v,\nabla\partial_x^{\alpha}a\rangle,
\end{align*}
and choosing $\delta$ sufficiently small, then \eqref{j1}-\eqref{j3} and \eqref{j4} imply that
\begin{align*}
\frac{\mbox{d}M(t)}{\mbox{d}t}+C_1(\|\nabla^2a\|_{H^1}^2+\|\nabla^2(v,h,m,\varepsilon)\|_{H^2}^2)\leq C\delta \|\nabla(a, v, h, m, \varepsilon)\|_{2}^2,
\end{align*}
where $C_1$ is a positive constant independent of $\delta$. Thus we arrive at the proof of the lemma.
\end{proof}

\section{Optimal convergence rates}
The optimal convergence rates can be proved by first improving the estimates given in Lemma \ref{lem3.5} and Lemma \ref{lem3.6} to the estimates on the $L^2$-norms of solutions to higher power and then letting the power tend to infinity. By the inequalities \eqref{inequality3.5} and \eqref{inequality3.6}, we have the following lemmas.
\begin{lem}\label{lem4.1}
Let $W=(U, h, m, \varepsilon)$ be the solution to the problem \eqref{ke-new}, then under the assumptions of Theorem \ref{thm2}, if $\delta>0$ is sufficiently small, then for any integer $n\geq 1$, and for some  $p\in[1,\frac65)$, it holds that
\begin{align}\label{inequality4.1}
\int_0^t(1+s)^l\|\nabla W(s)\|_2^{2n}\mbox{d}s\leq (CE_0)^{2n}+(C\delta)^{2n}\int_0^t(1+s)^l\|\nabla^2 W(s)\|_{H^1}^{2n}\mbox{d}s,
\end{align}
where $l=0, 1, \cdots, N=[2n(\frac{3}{2p}-\frac{1}{4})-2]$, the constant $E_0$ is given in Lemma \ref{lem3.5}.
\end{lem}
\begin{lem}\label{lem4.2}
Let $W=(U, h, m, \varepsilon)$ be the solution to the problem \eqref{ke-new}, then under the assumptions of Theorem \ref{thm2}, if $\delta>0$ is sufficiently small, then for any integer $n\geq 1$, and for some  $p\in[1,\frac65)$, it holds that
\begin{align}\label{inequality4.2}
\begin{split}
&(1+t)^lM(t)^{n}+n\int_0^t(1+s)^lM(s)^{n-1}\|\nabla^2W(s)\|_{H^1}^2\mbox{d}s\\
&\leq 2M(0)^{n}+(C_3E_0)^{2n}+8C_2n(\frac{3}{2p}
-\frac{1}{4})\int_0^t(1+s)^{l-1}M(s)^{n-1}\|\nabla^2W(s)\|_{H^1}^2\mbox{d}s,
\end{split}
\end{align}
where $l=0, 1, \cdots, N=[2n(\frac{3}{2p}-\frac{1}{4})-2]$, the constant $C_2$ is given in Lemma \ref{lem3.6}, $C_3$ is independent of $\delta$.
\end{lem}
\begin{rem}
Lemma \ref{lem4.1} and Lemma \ref{lem4.2} are similar to that in \cite{tong}. But Lemma \ref{lem4.1} and Lemma \ref{lem4.2} in this work hold for general $p$ with $p\in[1,\frac65)$. Note that lemmas hold only for $p=1$ in \cite{tong}. For completeness, we state the proofs of Lemma \ref{lem4.1} and Lemma \ref{lem4.2} as follows.
\end{rem}
\emph{Proof of Lemma \ref{lem4.1}.}  Fix any integer $n\geq1$. By taking \eqref{inequality3.5} to power $2n$ and multiplying it by $(1+t)^l$, $l=0,1,\cdots, N$, integrating the resulting inequality over $[0,t]$ gives that
\begin{align}\label{lemc4.2}
\begin{split}
&\int_0^t(1+\tau)^l\|\nabla W(\tau)\|_2^{2n}\mbox{d}\tau\leq (CE_0)^{2n}\int_0^t(1+\tau)^{-2n(\frac{3}{2p}-\frac{1}{4})+l}\mbox{d}\tau
\\
&+(C\delta)^{2n}\int_0^t(1+\tau)^l\big[\int_0^{\tau}(1+\tau-s)^{-(\frac{3}{2p}-\frac{1}{4})}\|\nabla W(s)\|_{H^2}\mbox{d}s\big]^{2n}\mbox{d}\tau.
\end{split}
\end{align}
It follows from the H\"{o}lder inequality that
\begin{align}\label{lemc4.3}
\begin{split}
&\big[\int_0^{\tau}(1+\tau-s)^{-(\frac{3}{2p}-\frac{1}{4})}\|\nabla W(s)\|_{H^2}\mbox{d}s\big]^{2n}
\\&\leq \big[\int_0^{\tau}(1+\tau-s)^{-r_1}(1+s)^{-r_2}\mbox{d}s\big]^{2n-1}\\
&\times \int_0^{\tau}(1+\tau-s)^{-\frac{4}{3}}(1+s)^l\|\nabla W(s)\|_{H^2}^{2n}\mbox{d}s,
\end{split}
\end{align}
where $$r_1=(\frac{3}{2p}-\frac{1}{4}-\frac{2}{3n})\frac{2n}{2n-1}$$
and $$r_2=\frac{l}{2n-1}.$$
Notice that
$\frac{3}{p}-\frac{11}{6}\leq r_1\leq \frac{3}{2p}-\frac{1}{4}$
and $r_2\in[0,r_1]$ for $n\geq1$ and $0\leq l\leq N=[2n(\frac{3}{2p}-\frac{1}{4})-2]$, from Lemma \ref{lem2.5},
one deduces that
\begin{align}\label{lemc4.4} \int_0^{\tau}(1+\tau-s)^{-r_1}(1+s)^{-r_2}\mbox{d}s\leq C_1(r_1,r_2)(1+\tau)^{-r_2}\leq C(1+\tau)^{-r_2},
\end{align}
where $C_1(r_1,r_2)$ given by \eqref{2.10} is bounded uniformly for $n\geq1$. Hence, \eqref{lemc4.2} together with \eqref{lemc4.3} and \eqref{lemc4.4} leads to
\begin{align}\label{lemc4.5}
\begin{split}
&\int_0^t(1+\tau)^l\|\nabla W(\tau)\|_2^{2n}\mbox{d}\tau\leq (CE_0)^{2n}\frac{1}{2n(\frac{3}{2p}-\frac{1}{4})-l-1}
\\
&+(C\delta)^{2n}\int_0^t(1+s)^l\|\nabla W(s)\|_{H^2}^{2n}\int_s^t(1+\tau-s)^{-\frac43}\mbox{d}\tau\mbox{d}s\\
&\leq (CE_0)^{2n}+(C\delta)^{2n}\int_0^t(1+s)^l\|\nabla W(s)\|_{H^2}^{2n}\mbox{d}s\\
&\leq (CE_0)^{2n}+(C\delta)^{2n}\int_0^t(1+s)^l(\|\nabla W(s)\|_{2}^{2n}+\|\nabla^2 W(s)\|_{H^1}^{2n})\mbox{d}s.
\end{split}
\end{align}
Here we have used the fact
$$2n(\frac{3}{2p}-\frac{1}{4})-l-1\geq2n(\frac{3}{2p}-\frac{1}{4})-2n(\frac{3}{2p}-\frac{1}{4})+2-1=1.$$
Thus if $\delta>0$ is sufficiently small such that $(C\delta)^{2n}\leq\frac12$ in the final inequality of \eqref{lemc4.5}, then \eqref{lemc4.5} implies  \eqref{inequality4.1}. We finish the proof of Lemma \ref{lem4.1}.
\qed

\emph{Proof of Lemma \ref{lem4.2}.} Multiplying \eqref{inequality3.6} by $n(1+t)^lM(t)^{n-1}$ for $l=0,1,\cdots, N$ and integrating it over $[0,t]$ give that
\begin{align}\label{lemcc4.7}
\begin{split}
&(1+t)^lM(t)^{n}+n\int_0^t(1+s)^lM(s)^{n-1}\|\nabla^2W(s)\|_{H^1}^2\mbox{d}s\\
&\leq M(0)^{n}+C\delta n\int_0^t(1+s)^lM(s)^{n-1}\|\nabla W(s)\|_{2}^2\mbox{d}s+l\int_0^t(1+s)^{l-1}M(s)^{n}\mbox{d}s.
\end{split}
\end{align}
For the second term on the right hand side of \eqref{lemcc4.7}, from the Young inequality, \eqref{energy} and Lemma \ref{lem4.1}, it holds that for any $\eta>0$,
\begin{align}\label{lemcc4.8}
\begin{split}
&\delta n\int_0^t(1+s)^lM(s)^{n-1}\|\nabla W(s)\|_{2}^2\mbox{d}s\\
&\leq\delta n\int_0^t(1+s)^l[\frac{n-1}{n}\eta M(s)^{n}+\frac{1}{n}\frac{1}{\eta^{n-1}}\|\nabla W(s)\|_{2}^{2n}]\mbox{d}s\\
&\leq\delta n C_2\eta\int_0^t(1+s)^l M(s)^{n-1}(\|\nabla W(s)\|_{2}^{2}+\|\nabla^2 W(s)\|_{H^1}^{2})\mbox{d}s\\
&+\delta\eta^{1-n}\int_0^t(1+s)^l\|\nabla W(s)\|_{2}^{2n}\mbox{d}s\\
&\leq \delta n C_2\eta\int_0^t(1+s)^l M(s)^{n-1}\|\nabla W(s)\|_{2}^{2}\mbox{d}s\\
&+\delta n C_2\eta\int_0^t(1+s)^l M(s)^{n-1}\|\nabla^2 W(s)\|_{H^1}^{2}\mbox{d}s\\
&+\delta\eta^{1-n}[(CE_0)^{2n}+(C\delta)^{2n}\int_0^t(1+s)^l\|\nabla^2 W(s)\|_{H^1}^{2n}\mbox{d}s]\\
&\leq \delta\eta^{1-n}(CE_0)^{2n}+\delta n C_2\eta\int_0^t(1+s)^l M(s)^{n-1}\|\nabla W(s)\|_{2}^{2}\mbox{d}s\\
&+\delta n C_2\eta\int_0^t(1+s)^l M(s)^{n-1}\|\nabla^2 W(s)\|_{H^1}^{2}\mbox{d}s\\
&+\delta\eta^{1-n}(C\delta)^{2n}C_2^{n-1}\int_0^t(1+s)^lM(s)^{n-1}\|\nabla^2 W(s)\|_{H^1}^{2}\mbox{d}s.
\end{split}
\end{align}
Choosing $\eta=\frac{1}{2C_2}$ in \eqref{lemcc4.8}, it holds that
\begin{align}\label{lemcc4.9}
\begin{split}
&\delta n\int_0^t(1+s)^lM(s)^{n-1}\|\nabla W(s)\|_{2}^2\mbox{d}s\\
&\leq2\delta (2C_2)^{n-1}(CE_0)^{2n}+\delta n[1+\frac{2}{n}(C\delta)^{2n}(2C_2^2)^{n-1}]\int_0^t(1+s)^lM(s)^{n-1}\|\nabla^2 W(s)\|_{H^1}^{2}\mbox{d}s
\\
&\leq\delta (CE_0)^{2n}+\delta n[1+(C\delta)^{2n}]\int_0^t(1+s)^lM(s)^{n-1}\|\nabla^2 W(s)\|_{H^1}^{2}\mbox{d}s.
\end{split}
\end{align}
Thus if $\delta>0$ is sufficiently small such that $C\delta\leq1$ in \eqref{lemcc4.9}, then $(C\delta)^{2n}\leq 1$ for any $n\geq 1$. And from \eqref{lemcc4.9}, it is easy to show that
 \begin{align}\label{lemcc4.10}
\begin{split}
&\delta n\int_0^t(1+s)^lM(s)^{n-1}\|\nabla W(s)\|_{2}^2\mbox{d}s
\\
&\leq(CE_0)^{2n}+2\delta n\int_0^t(1+s)^lM(s)^{n-1}\|\nabla^2 W(s)\|_{H^1}^{2}\mbox{d}s.
\end{split}
\end{align}
Similarly, we can estimate for the third term on the right hand side of \eqref{lemcc4.7} as
 \begin{align}\label{lemcc4.13}
 \begin{split}
&l\int_0^t(1+s)^{l-1}M(s)^{n}\mbox{d}s
\\&\leq(CE_0)^{2n}+\delta n(C\delta)^{2n-1}\int_0^t(1+s)^lM(s)^{n-1}\|\nabla^2 W(s)\|_{H^1}^{2}\mbox{d}s\\
&+2lC_2\int_0^t(1+s)^{l-1}M(s)^{n-1}\|\nabla^2 W(s)\|_{H^1}^{2}\mbox{d}s,
\end{split}
\end{align}
which together with \eqref{lemcc4.7} and \eqref{lemcc4.10} arrives at
\begin{align}\label{lemcc4.14}
\begin{split}
&(1+t)^lM(t)^{n}+n\int_0^t(1+s)^lM(s)^{n-1}\|\nabla^2W(s)\|_{H^1}^2\mbox{d}s\\
&\leq M(0)^{n}+(CE_0)^{2n}+\delta n[C+(C\delta)^{2n-1}]\int_0^t(1+s)^lM(s)^{n-1}\|\nabla^2 W(s)\|_{H^1}^{2}\mbox{d}s\\\
&+2lC_2\int_0^t(1+s)^{l-1}M(s)^{n-1}\|\nabla^2 W(s)\|_{H^1}^{2}\mbox{d}s.
\end{split}
\end{align}
Choose $\delta>0$ sufficiently small such that for any $n\geq 1$, it follows that
$$\delta[C+C\delta)^{2n-1}]\leq \frac12,$$
then it follows from \eqref{lemcc4.14} that
\begin{align*}
\begin{split}
&(1+t)^lM(t)^{n}+n\int_0^t(1+s)^lM(s)^{n-1}\|\nabla^2W(s)\|_{H^1}^2\mbox{d}s\\
&\leq 2M(0)^{n}+(C_3E_0)^{2n}+4lC_2\int_0^t(1+s)^{l-1}M(s)^{n-1}\|\nabla^2 W(s)\|_{H^1}^{2}\mbox{d}s,
\end{split}
\end{align*}
which gives \eqref{inequality4.2} because $l\leq N\leq 2n(\frac{3}{2p}-\frac{1}{4})$. Thus we complete the proof of Lemma \ref{lem4.2}.\qed

\emph{Proof of Theorem \ref{thm2}.} Let $\delta>0$ be small enough such that Lemma \ref{lem4.2} holds for any $n\geq2$. For any fixed integer $n\geq 2$, from Lemma \ref{lem4.2}, we get that the inequality \eqref{inequality4.2} holds for any $l=0,1,\cdots, N$. When $l=1$, \eqref{inequality4.2} reads
\begin{align}\label{lemcc4.15}
\begin{split}
&(1+t)M(t)^{n}+n\int_0^t(1+s)M(s)^{n-1}\|\nabla^2W(s)\|_{H^1}^2\mbox{d}s\\
&\leq 2M(0)^{n}+(C_3E_0)^{2n}+8C_2n(\frac{3}{2p}
-\frac{1}{4})\int_0^tM(s)^{n-1}\|\nabla^2W(s)\|_{H^1}^2\mbox{d}s.
\end{split}
\end{align}
It follows from \eqref{propc1.7} in Proposition \ref{thm1} that
\begin{align*}
\begin{split}
\int_0^tM(s)^{n-1}\|\nabla^2W(s)\|_{H^1}^2\mbox{d}s
&\leq [\sup_{s\geq0}M(s)]^{n-1}\int_0^t\|\nabla^2W(s)\|_{H^1}^2\mbox{d}s\\
&\leq (C_2C_0\delta^2)^{n-1}C_0\delta^2\leq(C_2C_0\delta^2)^{n}.
\end{split}
\end{align*}
which together with \eqref{lemcc4.15} shows that
\begin{align}\label{lemcc4.17}
\begin{split}
&(1+t)M(t)^{n}+n\int_0^t(1+s)M(s)^{n-1}\|\nabla^2W(s)\|_{H^1}^2\mbox{d}s\\
&\leq 2M(0)^{n}+(C_3E_0)^{2n}+8C_2n(\frac{3}{2p}
-\frac{1}{4})(C_2C_0\delta^2)^{n}.
\end{split}
\end{align}
For $1\leq l\leq N$, by induction one can arrive at
\begin{align}\label{lemcc4.16}
\begin{split}
&(1+t)^lM(t)^{n}+n\int_0^t(1+s)^lM(s)^{n-1}\|\nabla^2W(s)\|_{H^1}^2\mbox{d}s\\
&\leq [2M(0)^{n}+(C_3E_0)^{2n}]\sum_{i=1}^l[8C_2(\frac{3}{2p}
-\frac{1}{4})]^i+n[8C_2(\frac{3}{2p}
-\frac{1}{4})]^l(C_2C_0\delta^2)^n.
\end{split}
\end{align}
In fact, suppose that \eqref{lemcc4.16} holds for $1\leq l\leq N-1$. Then from \eqref{inequality4.2}, it holds that
\begin{align}\label{lemcc4.19}
\begin{split}
&(1+t)^{l+1}M(t)^{n}+n\int_0^t(1+s)^{l+1}M(s)^{n-1}\|\nabla^2W(s)\|_{H^1}^2\mbox{d}s\\
&\leq 2M(0)^{n}+(C_3E_0)^{2n}+8C_2n(\frac{3}{2p}
-\frac{1}{4})\int_0^t(1+s)^{l}M(s)^{n-1}\|\nabla^2W(s)\|_{H^1}^2\mbox{d}s\\
&\leq [2M(0)^{n}+(C_3E_0)^{2n}]+8C_2(\frac{3}{2p}
-\frac{1}{4})\{[2M(0)^{n}+(C_3E_0)^{2n}]\sum_{i=1}^l[8C_2(\frac{3}{2p}
-\frac{1}{4})]^{i-1}\\
&+n[8C_2(\frac{3}{2p}
-\frac{1}{4})]^l(C_2C_0\delta^2)^n\}\\
&\leq [2M(0)^{n}+(C_3E_0)^{2n}]\sum_{i=1}^{l+1}[8C_2(\frac{3}{2p}
-\frac{1}{4})]^{i-1}
+n[8C_2(\frac{3}{2p}
-\frac{1}{4})]^{l+1}(C_2C_0\delta^2)^n,
\end{split}
\end{align}
which combining with \eqref{lemcc4.17} gives that \eqref{lemcc4.16} holds for any $1\leq l\leq N$.

Specially,
\begin{align*}
(1+t)^NM(t)^n\leq [2M(0)^n+(C_3E_0)^{2n}]\frac{[8C_2(\frac{3}{2p}
-\frac{1}{4})]^N-1}{8C_2(\frac{3}{2p}
-\frac{1}{4})-1}+n[8C_2(\frac{3}{2p}
-\frac{1}{4})]^N(C_2C_0\delta^2)^n.
\end{align*}
Note that
$$2n(\frac{3}{2p}
-\frac{1}{4})-3\leq N=[2n(\frac{3}{2p}
-\frac{1}{4})-2]\leq 2n(\frac{3}{2p}
-\frac{1}{4})-1.$$
It is not difficult to prove that
$$(1+t)^{2n(\frac{3}{2p}
-\frac{1}{4})-3}\leq C^{2n(\frac{3}{2p}
-\frac{1}{4})}[M(0)^n+E_0^{2n}+\delta^{2n}],$$
which gives that
\begin{align*}
M(t)^{\frac12}\leq C[M(0)^n+E_0^{2n}+\delta^{2n}]^{\frac{1}{2n}}(1+t)^{-(\frac{3}{2p}-\frac{1}{4})+\frac{3}{2n}}.
\end{align*}
Since $M(0)$, $E_0$ and $\delta$ are independent of $n$,
one gets
\begin{align*}
[M(0)^n+E_0^{2n}+\delta^{2n}]^{\frac{1}{2n}}\rightarrow\max\{\sqrt{M(0)}, E_0, \delta\},\ \mbox{as}\  n\rightarrow\infty.
\end{align*}
The above relation implies that
\begin{align*}
M(t)^{\frac12}\leq C\max\{\sqrt{M(0)}, E_0, \delta\}(1+t)^{-(\frac{3}{2p}-\frac{1}{4})},
\end{align*}
that is,
\begin{align*}
\|\nabla W(t)\|_{H^2}\leq C\max\{\sqrt{M(0)}, E_0, \delta\}(1+t)^{-\sigma(p,2;1)},
\end{align*}
which together with Lemma \ref{2.2} implies \eqref{decay2} and \eqref{decay3}.

Now, estimate for \eqref{decay1}. For this purpose, applying \eqref{decay3}, Lemma \ref{lem1} and Lemma \ref{lem2} leads to
\begin{align*}
\|W(t)\|_2&\leq CE_0(1+t)^{-\sigma(p,2;0)}+C\int_0^t(1+t-s)^{-\sigma(p,2;0)}(\|F(W)\|_p+\|F(W)\|_2)\mbox{d}s\\
&\leq CE_0(1+t)^{-\sigma(p,2;0)}+C\delta\int_0^t(1+t-s)^{-\sigma(p,2;0)}\|\nabla W(s)\|_{H^1}\mbox{d}s\\
&\leq CE_0(1+t)^{-\sigma(p,2;0)}+C\delta\int_0^t(1+t-s)^{-\sigma(p,2;0)}(1+s)^{-\sigma(p,2;1)}\mbox{d}s\\
&\leq C(1+t)^{-\sigma(p,2;0)}.
\end{align*}
By interpolation, we have that \eqref{decay1} holds for any $2\leq q\leq 6$.

 For \eqref{decay4}, from \eqref{ke-new}, we get
 \begin{align*}
 \|\partial_tW(t)\|_2&\leq \|\gamma\mbox{div}v\|_2+\|(F_1, F_2,F_3,F_4,F_5)\|_2+\|-\gamma \nabla a+\lambda\Delta v+\lambda\nabla\mbox{div}v\|_2\\
 &+\|\lambda\Delta h\|\|_2+\|\lambda\Delta m\|_2+\|\lambda\Delta \varepsilon\|_2\\
 &\leq C(\|\nabla(a,v,h,m,\varepsilon)\|_2+\|\nabla^2(v,h,m,\varepsilon)\|_2)\\
 &\leq CE_0(1+t)^{-\sigma(p,2;1)}.
 \end{align*}
Thus, \eqref{decay4} is proved. The proof of Theorem \ref{thm2} is complete.
\qed

\end{document}